\documentclass[11pt]{elsarticle}
\usepackage{amsmath,amssymb,amsthm}
\usepackage{hyperref}
\usepackage{mathrsfs}
\usepackage{graphicx}
\usepackage{tikz}
\usepackage{enumerate}
\usepackage{float}
\usepackage{etoolbox}
\BeforeBeginEnvironment{enumerate}{\normalfont}

\theoremstyle{plain}
\newtheorem{theorem}{Theorem}[section]
\newtheorem{lemma}[theorem]{Lemma}

\newtheorem{proposition}[theorem]{Proposition}
\newtheorem{corollary}[theorem]{Corollary}

\theoremstyle{definition}
\newtheorem{definition}[theorem]{Definition}
\newtheorem{remark}[theorem]{Remark}
\newtheorem{example}[theorem]{Example}

\newcommand{\I}{\item}
\newcommand{\II}{\begin{enumerate}}
\newcommand{\III}{\end{enumerate}}

\newcommand{\dda}{\mathord{\mbox{\makebox[0pt][l]{\raisebox{-.4ex}{$\downarrow$}}$\downarrow$}}}

\newcommand{\ua}{\mathord{\uparrow}}
\newcommand{\da}{\mathord{\downarrow}}
\newcommand{\rom}[1]{\rm{\uppercase\expandafter{\romannumeral #1}}}

\makeatletter
\def\ps@pprintTitle{%
	\let\@oddhead\@empty
	\let\@evenhead\@empty
	\def\@oddfoot{\reset@font\hfil\thepage\hfil}
	\let\@evenfoot\@oddfoot
}
\makeatother

\begin{document}
	
	\begin{frontmatter}
		
		\title{Sober topologies on a set}

		\author{Xiangrui Li}
		\ead{17731912783@163.com}
		\author{Qingguo Li}
		\ead{liqingguoli@aliyun.com}
		\address{School of Mathematics, Hunan University, Changsha, Hunan, 410082, China}
		\author{Dongsheng Zhao\corref{a1}}
		\ead{dongsheng.zhao@nie.edu.sg}
		\address{Mathematics and Mathematics Education, National Institute of Education, Nanyang Technological University, Singapore}
		\tnotetext[a]{This Research Supported by the National Natural Science Foundation of China(12231007)}
		\cortext[a1]{Corresponding author.}
		\begin{abstract}
			The collection of all topologies on a set $X$ forms a complete lattice with respect to the inclusion order, which have been investigated by many researchers. Sobriety is  one of the core and extensively studied properties  in non-Hausdorff topology. This property plays a crucial role in characterizing the spectral spaces of commutative rings and topological spaces determined by their lattices of open sets.  In this paper, we  investigate the statute of  sober topologies in the complete lattice of all topologies on a given set. The main results to be proved include: (1) every $T_1$ topology is the join of some sober topologies; (2) every topology is the meet of some sober topologies; (3) the set of all sober topologies is directed complete; (4) every Alexanderoff - discrete topology is the meet of some sober Alexanderoff - discrete topologies; (5) the minimal sober topologies are exactly the Scott topologies of
sup-complete chains;  (6) an example will be constructed to show that the intersection of a decreasing sequence of Hausdorff topologies need not be sober.
\end{abstract}
		
	\end{frontmatter}
	
	In mathematics, Lattice Theory often shows to be useful in the study of the set of objects of a given type. Once the considered objects form a (complete) lattice $\mathcal{L}$, one immediately have some natural problems about a sub class $\mathcal{B}\subseteq \mathcal{L}$: (i) Is $\mathcal{B}$
closed under finite (arbitrary) joins or meets? (ii) which objects in $\mathcal{L}$ are the joins (meets, resp.) of objects in $\mathcal{B}$? (iii) what are the maximal (minimal, resp.) objects in $\mathcal{B}$, etc.  Garrett Birkhoff \cite{birkhoff-1936} first considered the lattice  of all topologies on a set with respect to  the inclusion order. Such lattices have then been extensively  studied by people  from different point of views. Some classic results include (i) the compact Hausdorff topologies are minimal Hausdorff; (ii) the lattice of all topologies on a set is complemented; (iii) the co-finite topology is the unique minimal $T_1$ topology on a set, etc.

In their paper \cite{larson-andima}, Larson and Andima presented a fairly complete survey on  the  properties of  $\Lambda(X)$ of all topologies on a fixed set $X$ satisfying a property $p$ for various $p$. See also \cite{cameron-1977} for a survey on  maximal topologies and \cite{alas-2006}\cite{berri-1964}\cite{bourbaki-1941}\cite{cameron-1971}\cite{kalapodi-tzannes-2017}\cite{mercado-aurichi-2019}\cite{ramanathan-1927} for additional results on maximal and minimal topologies.  For other work on this topic, see \cite{baldovino-costantini-2009}\cite{birkhoff-1936}\cite{steiner-1966}\cite{watson-1994}.
	
A topological space $X$ is sober if its every irreducible closed set is the closure of a unique singleton set.
Sober topologies appeared naturally in several parts of mathematics. The well known result by M. Hochster states that a topological space is homeomorphic to the spectral space of a commutative ring if and only if  (i) $X$ is sober, (ii) the compact open sets form a base of $X$, (iii) the intersection of two compact open sets is compact and (iv) $X$ is compact\cite{hochster-1969}.  The sobriety was also used in characterizing the topological spaces which are determined by their open set lattices \cite{drak-thron}. Moreover, the category of all sober spaces is reflective in the category of all $T_0$ space. With the emerging and development of domain theory, sobriety  has become one of the most extensively studied  non-Hausdorff properties, in particular for the Scott spaces of directed complete posets \cite{miao-li-xi-zhao-2023}\cite{miao-li-xi-zhao-2021}\cite{xu-xi-zhao-2021}\cite{xu-zhao}.

In this paper, we study the class  of all sober topologies in the lattice of all topologies on a fixed set. The main results to be proved include: (1) every $T_1$ topology is the join (supremum) of some sober topologies; (2) every topology is the meet of some sober topologies; (3) the set of all sober topologies is directed complete; (4) the minimal sober topologies  are exactly the Scott topologies of sup-complete chains; (5) every Alexanderoff - discrete topology is the meet of some sober Alexanderoff - discrete topologies; (6) the meet of a decreasing sequence of Hausdorff topologies need not be a sober topology.
	
	\section{Preliminary}

	Given a set $X$, the set $\mathrm{T}(X)$ of all topologies on $X$, equipped with the set inclusion order $\subseteq$,  is a complete lattice.  For any $\mathcal{A}\subseteq \mathrm{T}(X)$,
	$$\bigwedge \mathcal{A}=\inf \mathcal{A}=\bigcap \mathcal{A},$$
	and
$$\bigvee \mathcal{A}$$
 equals the topology with $\bigcup \mathcal{A}$ as a subbase. Usually, $\bigvee \mathcal{A}$ is called the topology generated by $\bigcup \mathcal{A}$.
	
	In the following, for a subset $A$ in a topological space $(X, \tau)$, we shall use $cl(A)$ or $\overline{A}$ to denote the closure of $A$. We may also use $cl_{\tau}(A)$  to denote the closure of $A$ if we wish  to indicate explicitly the topology $\tau$. Also $\Gamma(X, \tau)$ (or $\Gamma_{\tau}(X)$) will be used to denote the set of all closed sets of  $(X, \tau)$, called the {\sl co-topology} of $(X, \tau)$.
	
A non-empty subset  $A$ in a topological space $X$ is {\sl  irreducible} if for any closed sets $F, G$ in $X$, $A\subseteq F\cup G$ implies
	either $A\subseteq F$ or $A\subseteq G$.
	
	We shall use $\mathrm{Irr}(X, \tau) $ (or just $\mathrm{Irr}(X)$) to denote the set of all  irreducible sets of space $(X, \tau)$, and use $\mathrm{Irr}_c(X, \tau)$ (or just $\mathrm{Irr}_c(X)$) to denote the set of all closed irreducible sets in $(X, \tau)$.
	
	\begin{remark}\label{basic properties of irr sets}
		\II
		\I[(1)] If $A$ is an irreducible set, then so is $cl(A)$.
        \I[(2)] $A$ is irreducible if and only if for any open sets $U$ and  $V$, $A\cap U\not=\emptyset $ and $ A\cap V\not=\emptyset$ imply
        $A\cap U\cap V\not=\emptyset$.
		\I[(3)] If $f: X\longrightarrow Y$ is a continuous function between two topological spaces and $A\subseteq X$ is irreducible, then $f(A)$ is an irreducible set in $Y$.
		\I[(4)] Every singleton set $\{x\}$ is irreducible. Thus every $cl(\{x\})$ is irreducible.
		\I[(5)] If $\tau_1$ and $\tau_2$ are two topologies on a set $X$ such that $\tau_1\subseteq \tau_2$, then
		$$\mathrm{Irr}(X, \tau_2) \subseteq \mathrm{Irr}(X, \tau_1). $$
		\III
	\end{remark}
	
	\begin{definition}
		A topological space $(X, \tau)$ is {\sl sober} if for every closed irreducible set $A$ in $X$, there is a unique $x\in X$  such that $A=cl(\{x\})$.
		
		In this case, $\tau$ is also called a {\sl sober} topology on $X$.
	\end{definition}
	
	By the uniqueness of $x$ in the above definition, it follows that every sober space is $T_0$.
	
    If $X$ is a sober space, then the function $f: X\longrightarrow \mathrm{Irr}_c(X)$, defined by $f(x)=cl(\{x\})$,  is a bijection, hence   $|X|=|\mathrm{Irr}_c(X)|$.

	\begin{lemma}\label{irr set in generated topo}
		Let $\{\tau_i: i\in I\}\subseteq \mathrm{T}(X)$ and $\tau=\bigvee\{\tau_i: i\in I\}$. Then for any $A\in \mathrm{Irr}(X,\tau)$,
		$$cl_{\tau}(A)=\bigcap\{cl_{\tau_i}(A): i\in I\}.$$
		
		In particular, for each $A\in \mathrm{Irr}_c(X, \tau)$,
		$$A=\bigcap\{cl_{\tau_i}(A): i\in I\}.$$
	\end{lemma}
	
	\begin{proof}
	Let $B=\bigcap\{cl_{\tau_i}(A): i\in I\}$. Since $\tau_i\subseteq \tau$ for each $i\in I$, $cl_{\tau}(A)\subseteq cl_{\tau_i}(A)$ for each $i\in I$.  Hence, $cl_{\tau}(A)\subseteq B$.
		
	For any  $x\in X - cl_{\tau}(A)$, as $X-cl_{\tau}(A)\in \tau$ and $\tau$ has $\bigcup\{\tau_i: i\in I\}$ as a subbase, there are $U_k\in\tau_{i_k} (k =1, 2, \cdots, m)$ such that $x\in \bigcap\{U_k: k=1, 2, \cdots, m\}\subseteq X-cl_{\tau}(A)$. Then $A\subseteq cl_{\tau}(A) \subseteq U_{1}^c\cup U_{2}^c\cup\cdots\cup U_{m}^c$, where $U_{k}^{c}$ is the complement of $U_{k}$.
 Note that for each $k$, $U_k\in\tau$, thus $U_{k}^c$ is a closed set in $(X, \tau)$. By the irreducibility of $A$, $A\subseteq U_{k'}^c$ for some $k'$. Then, as $x\in U_{k'}\in \tau_{i_{k'}}$ and $A\cap U_{k'}=\emptyset$, we have that   $x\not\in cl_{\tau_{i_{k'}}}(A)$, implying $x\not\in B$.  Thus $B\subseteq cl_{\tau}(A)$, therefore $cl_{\tau}(A)=B$ as desired.
	\end{proof}

	Note that $cl_{\tau}(A)=\bigcap\{cl_{\tau_i}(A): i\in I\}$ may fail to be true if $A$ is not irreducible.

	\begin{remark}\label{finite space is sober}
		\II
		\I[(1)] Let $X$ be a finite $T_0$ space. Then $cl(\{x\})\not=cl(\{y\})$ for any $x, y\in X$ if $x\not=y$.
		For any $F\in \mathrm{Irr}_{c}(X)$, $F=\bigcup\{cl(\{x\}): x\in F\}$. Since $F\subseteq X$ is a finite set, we have $F=cl(\{x\})$ for some $x\in F$. It follows that $X$ is sober.
		\I[(2)] Every $T_2$ topology is sober.
		\III
	\end{remark}

For any $T_0$ space $(X, \tau)$, the {\sl specialization order} $\le_{\tau}$, which is  a partial order on $X$,  is defined by
$$x\le_{\tau} y \mbox{ if and only if } x\in cl(\{y\}).$$

A nonempty subset $D$ of a poset $P$ is {\sl directed } if  for any two elements $x, y\in D$, there is a $z\in D$ such that $x\le z, y\le z$.
A poset $P$ is {\sl directed complete }  if for any directed subset $D\subseteq P$, $\bigvee D$ exists in $P$.
A directed complete poset will also be called a dcpo.

A subset $A$ of a poset $P$ is an upper (lower, resp.) set if $A=\ua A=\{x\in P: x\geq a \mbox{ for some } a\in A\}$ ($A=\da A=\{x\in P: x\leq a \mbox{ for some } a\in A\}$, resp.)

A subset $U$ of a poset $P$ is Scott open if  (i) $U$ is an upper set, and (ii) for any directed set $D$ with $\sup D$ existing, $\sup D\in U$ implies $D\cap U\not=\emptyset$.
All Scott open sets of a poset $P$ form a topology on $P$, called the Scott topology of $P$ and is denoted by $\sigma(P)$.

For any poset $(P, \le)$, the specialization order $\le_{\sigma}$ of $(P, \sigma(P))$ coincides with the partial order $\le$  on $P$.

A subset $A$ of a space $(X, \tau)$ is saturated if it equals the intersection of all open sets containing $A$.  A $T_0$  space $X$ is  well-filtered if for any open set $U$ and filter $\mathcal{F}$ of saturated compact subsets of $X$,
$$\bigcap\mathcal{F}\subseteq U \mbox{ implies } F\subseteq U \mbox{ for some } F\in\mathcal{F}.$$
 	
A $T_0$ space $(X, \tau)$ is called a d-space (or monotone convergence space) if  $(X, \le_{\tau})$ is a dcpo  and every $U\in \tau$ is a Scott open set of $(X, \le_{\tau})$ (that is, $\tau\subseteq \sigma(X, \le_{\tau})$). By \cite{li-yuan-zhao-2020}, $(X, \tau)$ is a d-space if and only if for any directed set $D$ of $(X, \le_{\tau})$ and open set $U$,
$$\bigcap\{\uparrow d: d\in D\}\subseteq U \mbox{ implies } d\in U \mbox{ for some } d\in D,$$
here $\ua  d=\ua\{d\}$.

Every sober space is well-filtered and every well-filtered space is a d-space. For any d-space $(X, \tau)$, $(X,\le_{\tau})$ is a dcpo.

See \cite{Gier-2003} and \cite{Goubault-2013} for more about Scott topologies, sober spaces, well-filtered spaces and d-spaces.

Let $x, y$ are elements in a poset $P$. We say that $x$ is \emph{way-below}  $y$, in symbols $x\ll y$, if for every directed subset $D\subseteq P$ for which $\sup D$ exists,  $y\leqslant \sup D$  implies the existence of a $d\in D$ with $x\leqslant d$.
	For each $x\in P$, let $\dda x=\{y\in P: y\ll x\}$.
	
	A dcpo $P$ is called a {\sl domain } if for each $x\in P$, $\dda x$ is a directed set and $x=\bigvee \dda x$.
	
    For every domain $P$, $(P, \sigma(P))$ is sober \cite{Gier-2003} and \cite{Goubault-2013}.

	\section{Joins  of sober topologies}
	
	In  this section, we consider  the following problem: Which topologies $\tau\in \mathrm{T}(X)$ are the join of sober topologies?
 The main result is that every $T_1$ topology is such a topology. There is a $T_0$ topology which is not the join of sober topologies.
	
	For any set $X$, we shall use $\mathrm{T}_{sob}(X)$ to denote  the set of all sober topologies on  $X$.

	\begin{proposition}
		If $\tau\in \mathrm{T}(X)$ is the join of a finite number of sober topologies, then $|X|=|Irr_c(X, \tau)|,$ here $|X|$ is the cardinality of $X$.
	\end{proposition}
	\begin{proof}
		Let $\tau=\bigvee\{\tau_i: i\in D\}$ with $D=\{1, 2, \cdots, m\}$ and each $\tau_i$ sober.
		
		Since $\tau$ is finer than the  sober topology $\tau_1$, which is $T_0$, $\tau$ is a $T_0$ topology.
		
		If $X$ is a finite set, then by Remark \ref{finite space is sober},  $(X, \tau)$ is sober, thus
		$|X|=|\{cl(\{x\}): x\in X\}|=|\mathrm{Irr}_c(X, \tau)|.$
		
		Next, we assume that $X$ is an infinite set.
		
		Since $X$ is $T_0$, for any $x, y\in X$, $x=y$ if and only if $cl(\{x\})=cl(\{y\})$.
Hence,  $|X|=|\{cl(\{x\}): x\in X\}| \le |\mathrm{Irr}_c(X, \tau)|$ because  $cl(\{x\})\in \mathrm{Irr}_c(X, \tau)$  for each $x\in X$.
		
		For each $A\in \mathrm{Irr}_c(X, \tau)$, by Lemma \ref{irr set in generated topo},
		$$A=\bigcap\{cl_{\tau_i}(A): i\in D\}.$$
		
		Also $cl_{\tau_i}(A)\in \mathrm{Irr}_c(X, \tau_i)$ for each $i\in D$, by Remark \ref{basic properties of irr sets}(4).
		Thus there is an injective  $\phi: \mathrm{Irr}_c(X, \tau) \longrightarrow  \prod_{i\in D}\mathrm{Irr}_c(X, \tau_i)$ defined by
		$$\phi(A)=(cl_{\tau_1}(A), cl_{\tau_2}(A), \cdots, cl_{\tau_m}(A)),$$
		for each $A\in Irr_c(X, \tau)$.
		
		So $|X|\le |Irr_c(X, \tau)|\le |\prod_{i\in D}\mathrm{Irr}_c(X, \tau_i)|=|\mathrm{Irr}_c(X, \tau_1)|\times |\mathrm{Irr}_c(X, \tau_2)|\times\cdots\times |\mathrm{Irr}_c(X, \tau_m)|.$
		
		Since each $(X, \tau_i)$ is sober, $|X|=|\mathrm{Irr}_c(X, \tau_i)|$.
		Therefore,  $ |\mathrm{Irr}_c(X, \tau_1)|\times |\mathrm{Irr}_c(X, \tau_2)|\times\cdots\times |\mathrm{Irr}_c(X, \tau_m)|=|X|^m=|X|$, because $|X|$ is infinite.
		All these together deduce that
		$$|X|=|\mathrm{Irr}_c(X, \tau)|.$$
	\end{proof}
	
	\begin{example}\label{tau 1}
		Let  $\tau_{cof}$ be the co-finite topology on the set $\mathbb{N}$ of all positive integers ($U\in \tau_{cof}$ if and only if either $U=\emptyset$ or $\mathbb{N} - U$ is a finite set).

		Then $(\mathbb{N}, \tau_{cof})$ is a $T_1$ space. As $\mathbb{N}\in \mathrm{Irr}_{c}(X, \tau_{cof})$ and $\mathbb{N}\not=cl(\{x\})$ for any $x$, $(\mathbb{N}, \tau_{cof})$ is not sober.

		Let $\tau_{1}$ be the topology on $\mathbb{N}$ such that $U\in \tau_1$ iff either $U=\emptyset$ or $1\in U$ and $\mathbb{N} - U$ is a finite set.
		
		Let $\tau_{2}$ be the topology on $\mathbb{N}$ such that $U\in \tau_2$ iff either $U=\emptyset$ or $2\in U$ and $\mathbb{N} - U$ is a finite set.
		
		Then we can verify that both $\tau_1$ and $\tau_2$ are sober topologies and  $\tau_{cof}=\tau_1\vee \tau_2$.
Thus the  join of finite numbers of sober topologies need not be sober.
	\end{example}
	
The example below shows that not every topology is the join of sober topologies.
	\begin{example}
Let $\mathbb{R}$ be the set of all real numbers and $\tau=\{(r, +\infty): r\in \mathbb{R}\}\cup \{\mathbb{R}\}$ be the upper topology on $\mathbb{R}$. Then $\tau$ is a minimal $T_0$ topology on $\mathbb{R}$ by a characterization of minimal $T_0$ topologies given in \cite{larson-1969}. Hence, any topology strictly smaller than $\tau$ is not $T_0$, thus not sober. It follows that this $T_0$ topology  $\tau$
is not the join of sober topologies.

In general, any non-sober minimal $T_0$ topology is not the join of sober topologies.

	\end{example}
	
	
	However,  if $\tau$ is a $T_1$ topology, then it is the join of  sober topologies.
	
	\begin{theorem}\label{t1 topologies are joins of sober topologies}
		For any $T_1$ topology $\tau\in \mathrm{T}(X)$, there is a set $\mathcal{A}\subseteq \mathrm{T}_{sob}(X)$ such that
		$$\tau = \bigvee \mathcal{A}.$$
	\end{theorem}
	\begin{proof}
		We only need to consider the non-trivial  case where $(X, \tau)$ is not sober, thus $X$ must be an infinite set.

		(1) For each closed proper nonempty subset $A\subseteq X$, choose two points $x_A\in A$ and $y_A\not\in A$.
		Let
		$$\tau_{A}=\{U\subseteq X: x_A, y_A\in U \mbox{ and } U^c \mbox{ is finite }\} \cup \{A^c\cap V: V\subseteq X~\mathrm{with}~ x_A, y_A\in V~ \mathrm{ and }~ V^c \mbox{ is finite}\}.$$
One easily verify that $\tau_A$ is indeed a topology on $X$.

		The closed sets of $\tau_A$ are
		$$\{F: F\subseteq X - \{x_A, y_A\}~ \mathrm{ and }~ F \mbox{ is finite}\}\cup \{A\cup F: F\subseteq X-\{x_A, y_A\}~ \mathrm{ and }~F \mbox{ is finite}\}\cup \{X\}.$$
		
		In particular, $A=A\cup \emptyset$ is closed in $(X, \tau_A)$, that is,  $A^c\in \tau_A$.
		
		It is easily seen  that  $\tau_A$ is $T_0$ and
$$\mathrm{Irr}_c(X, \tau_A)=\{\{x\}: x\not\in\{x_A, y_A\}\}\cup \{A, X\}.$$
		
		If $x\not\in\{x_A, y_A\}$, then $\{x\}=cl_{\tau_A}(\{x\})$. Also $A=cl_{\tau_A}(\{x_A\})$ and $X=cl_{\tau_A}(\{y_A\})$.
		Therefore, $\tau_A$ is a sober topology.
\vskip 0.2cm
		Also, as $(X, \tau)$ is $T_1$, $\tau_A\subseteq \tau$ holds, therefore
$\bigvee\{\tau_A:  A^{c} \in \tau, A\not=X, A\not=\emptyset\}\subseteq \tau.$
\vskip 0.5cm

		(2) Now let $U\in\tau$.
	\vskip 0.2cm
		Case 1: $U=X~or~\emptyset$. Then $U\in\tau_{A}$ for each $\tau_{A}$.
	\vskip 0.2cm
		Case 2: $U\not=X$ and $U\not=\emptyset$.
		Let $A=U^{c}$.
		Then $A$ is a proper, nonempty closed set of $(X,\tau_A)$.
		By (1), $U=A^c\in\tau_A$.
		
		Hence $\tau\subseteq \bigvee\{\tau_{A}: A \mbox{ is a nonempty proper closed set of  } (X, \tau)\}.$

All these together then show that $\tau$ equals the join of these sober topologies $\tau_{A}'s$.
		
		The theorem is proved.
	\end{proof}
	
	\begin{remark}
\II
\I[(1)] The reader may wonder whether the above theorem can be strengthened to that every $T_1$ topology is the join of  $T_1$ sober topologies.
		Consider the co-finite topology $\tau_{cof}$ on the set $\mathbb{N}$ of all positive integers.
		Then $\tau_{cof}$ is the coarsest $T_1$ topology on $\mathbb{N}$. Thus if $\tau$ is a $T_1$ sober topology on $\mathbb{N}$, then $\tau_{cof}\subseteq \tau$ and $\tau_{cof}\not= \tau$. Hence $\tau_{cof}$ is not the join of $T_1$ sober topologies.
\I[(2)]
For each of the topology $\tau_A$ constructed in the proof of Theorem \ref{t1 topologies are joins of sober topologies}, $\tau_A$ actually possesses  several other properties, such as (i) $\tau_A$ is connected and locally connected; (ii) every subset of $X$ is compact.
\III
	\end{remark}

	An element $a$ of a complete lattice $L$ is  {\sl strongly irreducible} if for any $C\subseteq L, C\not=\emptyset$ and $a=\bigvee C$ imply $a=c$ for some $c\in C$ \cite{drak-thron}\cite{xul-zhao}.
	
	By Theorem \ref{t1 topologies are joins of sober topologies}, we easily deduce the following.
	
	\begin{corollary}
		If $\tau\in \mathrm{T}(X)$ is a $T_1$ topology and a strongly irreducible element of $\mathrm{T}(X)$, then $\tau$ is sober.
	\end{corollary}
	
A subset $C$ of a poset $(P, \le)$ is an upper set, if $x\le y$ and $x\in C$ imply $y\in C$, equivalently, if
$C=\ua\{y\in P: x\le y ~\mathrm{ for~some }~ x\in C\}$.

The following result shows that the subset of $\mathrm{T}(X)$ consisting of all $T_1$ sober topologies is an upper set of  $\mathrm{T}(X)$. In addition, all $T_D$ sober topologies is also an upper set of $\mathrm{T}(X)$.
	
	Recall that a space$(X, \tau)$ is $T_D$ if for each $x\in X$, $cl(\{x\}) -\{x\}$ is a closed set. Trivially, every $T_1$ spaces is $T_D$.

The following lemma should have been proved by other people. For reader's convenience,  we give a brief proof.

	\begin{lemma}\label{td}
		Let $(X,\tau)$ be a topological space. Then the following statements are equivalent.
		\II
		\I[(1)] $(X,\tau)$ is a $T_D$ space.
		
		\I[(2)] For any $x\in X$ and $A\subseteq X$, $cl(A)=cl(\{x\})$ implies $x\in A$.
		\III
	\end{lemma}
	\begin{proof}

Let $(X, \tau)$ be $T_D$. Assume that $cl(\{x\})=cl(A)$ for some $x\in X$ and $A\subseteq X$.
If $x\not\in A$, then $A\subseteq cl(\{x\})-\{x\}$.
Hence $cl(A)\subseteq cl(cl(\{x\})-\{x\})=cl(\{x\})-\{x\}\subsetneq cl(\{x\})$, a contradiction.
		
Now assume that $(X, \tau)$ satisfies (2).  Let $x\in X$ and $A=cl(\{x\})-\{x\}$. If $cl(A)\not=A$, then, as
$A\subseteq cl(A)\subseteq cl(\{x\})=A\cup \{x\}$, it follows that $x\in cl(A)$. By (2), $x\in A=cl(\{x\})-\{x\}$, which is not possible. Hence $cl(A)=A=cl(\{x\}) -\{x\}$, showing that $cl(\{x\}) -\{x\}$ is closed. Hence $(X, \tau)$ is $T_D$.
	\end{proof}
	
	\begin{proposition}\label{sober td spaces}
\II
\I[(1)] The set of all $T_D$ topologies on $X$ is an upper set of $\mathrm{T}(X)$.
\I[(2)]  If $\tau$ is a $T_D$ and sober topology on $X$, then for any $\mu\in \mathrm{T}(X)$, $ \tau\subseteq \mu $ implies $\mu$ is $T_D$ and sober.
\III
	\end{proposition}
	\begin{proof}

(1) Let $\tau$ be a $T_D$ topology and $\tau\subseteq\mu\in T(X)$. For any $x\in X$, $cl_{\mu}(\{x\})-\{x\}=cl_{\mu}(\{x\})\cap (cl_{\tau}(\{x\})-\{x\})$ by $cl_{\mu}(\{x\})\subseteq cl_{\tau}(\{x\})$. Now $cl_{\tau}(\{x\})-\{x\}\in\Gamma(X,\tau)\subseteq \Gamma(X,\mu)$, thus $cl_{\mu}(\{x\})-\{x\}=cl_{\mu}(\{x\})\cap (cl_{\tau}(\{x\})-\{x\}) \in \Gamma(X,\mu)$.
Hence $(X, \mu)$ is $T_D$.
		
(2)  By  (1), $\mu$ is $T_D$. Assume $C\in Irr_c(X,\mu)$. Then $C\in Irr(X,\tau)$ and so $cl_{\tau}(C)=cl_{\tau}(\{x\})$ for some $x\in X$. By Lemma \ref{td}(2), $x\in C$. We now prove that $C=cl_{\mu}(\{x\})$.
	Assume that $C\not=cl_{\mu}(\{x\})$, so $C - cl_{\mu}(\{x\})\not=\emptyset$. Then $C\subseteq cl_{\tau}(\{x\}) \subseteq (cl_{\tau}(\{x\})-\{x\})\cup cl_{\mu}(\{x\})$, and  $C\not\subseteq (cl_{\tau}(\{x\})-\{x\})$ and $C\not\subseteq cl_{\mu}(X)$. Note that $cl_{\tau}(\{x\})-\{x\}\in \Gamma(X, \tau)\subseteq \Gamma(X, \mu)$ and $cl_{\mu}(\{x\})\in \Gamma(X, \mu)$. These contradict to that $C\in Irr_c(X,\mu)$.
Thus $C=cl_{\mu}(\{x\})$. Hence $(X, \mu)$ is sober.

	\end{proof}
	
	\begin{corollary}\label{sober t1 spaces}
		If $\tau$ is a $T_1$ and  sober topology on $X$, then for any $\mu\in \mathrm{T}(X)$, $ \tau\subseteq \mu $ implies $\mu$ is $T_1$ and sober.
	\end{corollary}
	\begin{proof} The topology $\mu$ is $T_1$ as it is finer than a $T_1$ topology. Since  $\tau$ is $T_1$, it is $T_D$.  By Proposition \ref{sober td spaces},  $\mu$ is also sober.
	\end{proof}
	
\begin{corollary}
For any collection $\{\tau_i: i\in I\}$ of  $T_1$ and sober topologies on a set $X$,
$$\bigvee \{\tau_i: i\in I\}$$
is $T_1$ and sober.
\end{corollary}
		
	
	The following propositions can be found in \cite[Lemma 13.3]{kechris} and
	\cite[Lemma 72]{de2013quasi}, respectively.
	
	\begin{proposition}
		If $\tau_0$ is a quasi-polish topology on $X$ and $\{\tau_i: i\in \mathbb{N}\}$ is a set of quasi-polish topologies on $X$  such that $\tau_0\subseteq \tau_i$ for each $i\in \mathbb{N}$, then $\tau=\bigvee\{\tau_i: i\in \mathbb{N}\}$ is quasi-polish.
	\end{proposition}
	
	\begin{proposition}
	If $\tau_0$ is a quasi-polish topology on $X$ and $\{\tau_i: i\in \mathbb{N}\}$ is a set of polish topologies on $X$  such that $\tau_0\subseteq \tau_i$ for each $i\in \mathbb{N}$, then $\tau=\bigvee\{\tau_i: i\in\mathbb{N}\}$ is polish.
	\end{proposition}
	
\vskip 1cm
For sober topologies, we have a similar result where the countable  index set $\mathbb{N}$ can be replaced by any set.
	
	\begin{lemma}\label{irr sets of compa topo}
		Let $\tau_1, \tau_2$ be two sober topologies on a set $X$ with $\tau_1\subseteq \tau_2$.
		Then we have the following.
		\II
		\I[(1)] $\mathrm{Irr}_c(X, \tau_1) =\{cl_{\tau_1}(A): A\in \mathrm{Irr}_c(X, \tau_2)\}.$
		\I[(2)] For any $a, b\in X$, $cl_{\tau_1} (cl_{\tau_2}(\{a\}))=cl_{\tau_1}(\{b\})$ implies $a=b$.
		\III
	\end{lemma}
	\begin{proof}
	\rm{(1)} By Remark \ref{basic properties of irr sets} (1) (4),
$\{cl_{\tau_1}(A): A\in \mathrm{Irr}_c(X, \tau_2)\} \subseteq \mathrm{Irr}_c(X, \tau_1).$
			Now let $F\in Irr_c(X, \tau_1)$. Thus $F=cl_{\tau_1}(\{x_0\})$ for a unique $x_0\in X$.
			Now $cl_{\tau_2}(\{x_0\})\in \mathrm{Irr}_c(X, \tau_2)$ because $\{x_0\}$ is an irreducible set in $(X, \tau_2)$.
			So, $F\subseteq cl_{\tau_1}(cl_{\tau_2}(\{x_0\}))$ because $\{x_0\}\subseteq cl_{\tau_2}(\{x_{0}\})$.
			Also $cl_{\tau_2}(\{x_0\})\subseteq cl_{\tau_1}(\{x_0\})$, it follows that $F\subseteq cl_{\tau_1}(cl_{\tau_2}(\{x_0\}))\subseteq   cl_{\tau_1}(cl_{\tau_1}(\{x_0\}))= cl_{\tau_1}(\{x_0\})=F.$ Hence,
			$$ F=cl_{\tau_1}(K),$$
			where $K=cl_{\tau_2}(\{x_0\})\in \mathrm{Irr}_c(X, \tau_2)$.
			So the equation holds.
			
			\rm{(2)} If $cl_{\tau_1} (cl_{\tau_2}(\{a\}))=cl_{\tau_1}(\{b\})$ with $a, b\in X$,
			then, similar to the proof of (1), we deduce  that  $cl_{\tau_1}(\{a\}) = cl_{\tau_1} (cl_{\tau_2}(\{a\}))=cl_{\tau_{1}}(\{b\})$. Thus $a=b$ because $(X, \tau_1)$ is $T_0$.
			\end{proof}

	\begin{proposition}\label{2.10}
		If $\tau_0$ is a sober topology on $X$ and $\{\tau_i: i\in I\}$ is a set of sober topologies on $X$  such that $\tau_0\subseteq \tau_i$ for each $i\in I$, then $\tau=\bigvee\{\tau_i: i\in I\}$ is sober.
	\end{proposition}
	\begin{proof}
		Firstly, $(X, \tau)$ is $T_0$ because $\tau$ is finer than the $T_0$ topologies  $ \tau_{i} (i\in I).$
		
		Let $A\in \mathrm{Irr}_c(X, \tau)$. Then, again, by Remark \ref{basic properties of irr sets} (1)(4),
$cl_{\tau_j}(A)\in \mathrm{Irr}_c(X, \tau_j)$ for each $j\in I\cup\{0\}$ and  $A=\bigcap\{cl_{\tau_i}(A): i\in I\}$, by Proposition \ref{irr set in generated topo}. Since $\tau_j$ is sober for each $j\in I\cup\{0\}$,
		$cl_{\tau_j}(A)=cl_{\tau_j}(\{x_j\})$ for some $x_j\in X$. Now $cl_{\tau_0}(A)=cl_{\tau_0}(\{x_0\})\subseteq cl_{\tau_0}(cl_{\tau_i}(\{x_i\}))\subseteq cl_{\tau_0}(cl_{\tau_0}(\{x_0\}))=cl_{\tau_0}(\{x_0\})$, implying
$cl_{\tau_0}(cl_{\tau_i}(\{x_i\}))=cl_{\tau_0}(\{x_0\})$. By Lemma \ref{irr sets of compa topo}(2),  $x_0=x_i$ for each $i\in I$.
		
		Then $A=\bigcap\{cl_{\tau_i}(\{x_0\}): i\in I\}$.
		Note that the set $\{x_0\}$ is irreducible in $(X, \tau)$, by Lemma \ref{irr set in generated topo}, we also have
		
		$$cl_{\tau}(\{x_0\})=\bigcap\{cl_{\tau_i}(\{x_0\}): i\in I\}.$$
		At last, we have $A=cl_{\tau}(\{x_0\})$. The uniqueness of $x_0$ follows from that $\tau$ is $T_0$.  Therefore, $(X, \tau)$ is sober.	
	\end{proof}

 If $D$ is a directed subset of poset $P$ and $d_0\in D$, then $D_{d_0}=\{d\in D: d_0\le d\}$ is also a directed set and $\bigvee D=\bigvee D_{d_0}$ if $\bigvee D$ exists in $P$.
	
	Let  $\mathcal{B}=\{\tau_i: i\in I\}\subseteq T(X)$ be a directed set consisting of  sober topologies.  Fixed one $\tau_{i_0}$ and let
$\mathcal{B}_{i_0}=\{\tau_i: \tau_{i_0}\subseteq \tau_i\}$. Then by the previous remark, $\mathcal{B}_{i_0}$ is also a directed set and
$\bigvee \mathcal{B}=\bigvee\mathcal{B}_{0}$, here the suprema are taken in $\mathrm{T}(X)$. Since $\mathcal{B}_{0}$ has a bottom element $\tau_{i_0}$,  by Proposition \ref{2.10}, $\bigvee\mathcal{B}_{0}$ is sober. Hence  $\bigvee\mathcal{B}=\bigvee \mathcal{B}_{0}$ is sober.
	
	\begin{corollary}\label{directed join of sober topologies}
		The poset  $(\mathrm{T}_{sob}(X), \subseteq)$ of all sober topologies on a set $X$  is a dcpo.
	\end{corollary}

	For any $\tau\in \mathrm{T}(X)$, define
	$$\omega_{sob}(\tau)=min\{|\mathcal{A}|: \mathcal{A}\subseteq T_{sob}(X) \mbox{ and } \tau=\bigvee \mathcal{A}\},$$
	where $|\mathcal{A}|$ is the cardinality of the  set $\mathcal{A}$.
	
	Then $\omega_{sob}(\tau)=1$ if and only if $\tau$ is sober.
	
	By Example \ref{tau 1}, $\omega_{sob}(\tau_{cof})=2$, where $\tau_{cof}$ is the co-finite topology on $\mathbb{N}$.
	
	\begin{lemma}\label{join of sob topologies}
		Let $\{\tau_i: i\in I\}\subseteq T_{sob}(X)$ and $\tau=\bigvee\{\tau_i: i\in I\}.$
		Then
		$$|Irr_{c}(X, \tau)| \le |X|^{|I|}.$$
	\end{lemma}
	\begin{proof}
		Define the mapping
		$\phi: Irr_c(X, \tau)\longrightarrow \prod_{i\in I} Irr_c(X, \tau_i)$ by
		$$\phi(A)=(cl_{\tau_i}(A))_{i\in I},  A\in Irr_c(X, \tau).$$
		By Lemma \ref{irr set in generated topo}, $\phi$ is well defined and injective.
		For each $i\in I$, since $(X, \tau_i)$ is sober, we have  $Irr_c(X, \tau_i)=\{cl_{\tau_i}\{x\}: x\in X\}$. Thus $|Irr_c(X, \tau_i)|=|X|$.
		Therefore,
		$|Irr_c(X, \tau)|\le | \prod_{i\in I} Irr_c(X, \tau_i)|=|X|^{|I|}$, as desired.
	\end{proof}
	
	Now we give a topology $\tau$ such that $\omega_{sob}(\tau)$ is not finite.
	
	\begin{example}
		Let  $X=\mathbb{Q}\cap [0, 1]$ be the set of all rational numbers in $[0, 1]$.
		Let $\tau$ be the topology on $X$ whose family of closed sets equals
		$$\{F\subset X: F \mbox{ is finite }\}\cup \{([0, x]\cap \mathbb{Q})\cup F: x\in [0, 1], F\subseteq X \mbox{ is finite }\}.$$
		Then we can verify that
		$$Irr_c(X, \tau)=\{\{x\}: x\in X\}\cup \{[0, x]\cap\mathbb{Q}: x\in [0, 1]\}.$$
		Now $| Irr_c(X, \tau)|=|X| + |[0, 1]|=\aleph_1$.
		If $\tau=\bigvee \{\tau_i: i\in H\}$ with $H$ a finite set, then by  Lemma \ref{join of sob topologies},
		$\aleph_1 \le |X|^{|H|}=\aleph_0^{|H|}=\aleph_0$, a contradiction. It follows that  $\omega_{sob}(\tau)\ge \aleph_0$.
		
		In fact, $\Gamma(X,\tau)$ has a subbase $\{\{x\}: x\in X\}\cup\{[0,q]: q\in\mathbb{Q}\}\cup\{\emptyset\}$. And we have $\omega_{sob}(\tau)\le \aleph_0$, by the proof of Theorem \ref{t1 topologies are joins of sober topologies}.
		
		Hence, $\omega_{sob}(\tau)=\aleph_0$.
	\end{example}
	
	\section{Meets of sober topologies}
	In this section, we study the meets of sober topologies in $\mathrm{T}(X)$. The main result is that every  topology is the meet of a collection of sober topologies. Thus, the set $\mathrm{T}_{sob}(X)$ is meet dense in $\mathrm{T}(X)$.
	
	Note that for any $\mathcal{A}\subseteq \mathrm{T}(X)$,  $\inf\mathcal{A}=\bigwedge \mathcal{A}=\bigcap \mathcal{A}$.
	
	First, we show that every topology is the meet of some $T_0$ topologies. This result seems having been known for quite some time (we saw it in some online forum without proof),
 but we could not find a reliable source for it. For reader's convenience, we give a proof  here.
	
	Let $(M, \mu)$ be a topological space and $M\subseteq X$. Let $\tau_{M}$ be the topology on $X$ such that
	$$\Gamma(X, \tau_{M})= \{A\cup B: A\cap M=\emptyset \mbox{ and } B\in \Gamma(M, \mu)\}.$$
	By the definition of $\tau_M$, $(M,\mu)$ is a closed subspace of $(X, \tau_{M})$.
	
	\begin{lemma}\label{topology by subspace}
		With the above notations, we have the following statements.
		\II
		\I[(1)] $(M, \mu)$ is $T_0$ if and only if $(X, \tau_{M})$ is $T_0$.
		\I[(2)] $(M, \mu)$ is $T_1$ if and only if $(X, \tau_{M})$ is $T_1$.
		\I[(3)] $(M, \mu)$ is sober if and only if $(X, \tau_{M})$ is sober.
		\III
	\end{lemma}
	
  The space $(X, \tau_{M})$ is actually the direct sum $(X-M)\oplus M$ of the discrete space $X-M$ and the space $(M, \mu)$. Hence the three statements hold.
	
	\vskip 0.5cm
Given a topology $\tau\in \mathrm{T}(X)$, let $\thicksim_{\tau}$ be the equivalence  relation on $X$ defined by
	$$x \thicksim_{\tau} y \mbox{ if and only if } cl_{\tau}(\{x\}) =cl_{\tau}(\{y\}).$$

	Let $X/ \thicksim_{\tau} =\{[x_i]: i\in I\}$ be the collection of all distinct equivalence classes\\ determined by $ \thicksim_{\tau}$.
	Let $P=\{f\in X^{I}: f(i)\in [x_i] \mbox{ for each } i\in I\}.$
	
	For each $f\in P$, let $F_{f}: X\longrightarrow X$ be the mapping such that for each $x\in [x_i]_{\thicksim_{\tau}}, F_{f}(x)=f(i).$
	
	For each $f\in P$, let $\tau_{f}$ be the topology on $X$ such that
$$\Gamma(X, \tau)\cup \{F\subseteq X: F \mbox{ is finite and } f(i)\in F\subseteq [x_i]_{\thicksim_{\tau}} \mbox{for some } i\in I\}$$
 is a sub base of the co-topology of $\tau_f$.

\vskip 0.5cm

	Trivially, $\tau\subseteq \tau_{f}$. Also for any $a\in [x_i]_{\thicksim_{\tau}}$, $\{a, f(i)\}$ is the smallest closed set in $(X, \tau_f)$ containing $a$, hence $cl_{\tau_f}(\{a\})=\{a, f(i)\}$ when $a\not=f(i)$, and $cl_{\tau_f}(\{a\})=\{a\}$ if $a=f(i)$ .

	Note  that a topology $\mu\in \mathrm{T}(X, \tau)$ is $T_0$ iff  for any $a, b\in X$, $a\not=b$ implies
	$cl_{\mu}(\{a\})\not=cl_{\mu}(\{b\})$.
	
	\begin{lemma}\label{topology tauf}
		For each $f\in P$, $\tau_{f}$ is a $T_0$ topology.
	\end{lemma}
	\begin{proof}
		Let $a, b\in X$ with $a\not=b$.
		
		Case 1: $a\thicksim_{\tau} b$.
		
		Then  $\{a, b\}\subseteq [x_{i}]_{\thicksim_{\tau}} $ for some $i\in I$.
		
		If  $a=f(i)$ (or $b=f(i)$), then $cl_{\tau_{f}}(\{a\})=\{a\}~~  (cl_{\tau_f}(\{b\})=\{b\}$, resp.), thus $cl_{\tau_f}(\{a\})\not=cl_{\tau_f}(\{b\}).$
		
		If  $a\not=f(i)$ and $b\not=f(i)$, then one can verify that $cl_{\tau_f}(\{a\}) = \{f(i), a\}$, and as $b\not\in \{f(i), a\}$, again  $cl_{\tau_f}(\{a\})\not=cl_{\tau_f}(\{b\})$.
		
		Case 2: $a\thicksim_{\tau} b$ does not hold.
		
		Then either $a\not\in cl_{\tau}(\{b\})$ or $b\not\in cl_{\tau}(\{a\})$.
		As $cl_{\tau_{f}}(\{b\})\subseteq cl_{\tau}(\{b\})$, $cl_{\tau_{f}}(\{a\})\subseteq cl_{\tau}(\{a\})$, we have that
		$a\not\in cl_{\tau_{f}}(\{b\})$ or $b\not\in cl_{\tau_{f}}(\{a\})$, implying  $cl_{\tau_{f}}(\{b\})\not= cl_{\tau_{f}}(\{b\})$.
		
		All these together show that $(X, \tau_{f})$ is $T_0$.
	\end{proof}

	\begin{remark}\label{titlde operator}
		For each subset $A$ of a space $(X, \tau)$, let $\tilde{A}=\bigcup\{cl_{\tau}(\{x\}): x\in A\}$.
		Then
		\II
		\I[(1)] $A\subseteq \tilde{A}\subseteq cl_{\tau}(A)$;
		\I[(2)] $\tilde{(A\cup B)}=\tilde{A}\cup \tilde{B}$ for any $A, B\subseteq X$;
		\I[(3)] $\tilde{\tilde{A}}=\tilde{A}$.
		\III
		Hence $A\rightarrow \tilde{A}$ defines a closure operator on $\mathcal{P}(X)$.
		
	\end{remark}
	
	\begin{lemma}\label{every topology is the meets of t0 topo}
		For any $\tau\in \mathrm{T}(X)$, there is a collection $\mathcal{A}\subseteq \mathrm{T}(X)$ of $T_0$ topologies such that
		$$\tau =\bigcap \mathcal{A}.$$
	\end{lemma}
	\begin{proof}
		
		As before, let  $X/ \thicksim_{\tau} =\{[x_i]_{\thicksim_{\tau} }: i\in I\}$ and $P=\{f\in X^{I}: f(i)\in [x_i] \mbox{ for each } i\in I\}.$
		
		Let $\lambda=\bigcap\{\mu\in \mathrm{T}(X): \tau\subseteq \mu,~\mu \mbox{ is } T_0\}$.
		Then $\tau\subseteq \lambda$.
		
		Assume that $A\subseteq X$ with $A\not\in \Gamma(X,\tau)$, we show that $A\not\in \Gamma(X, \lambda)$.
		
		Case 1: $A\not=\tilde{A}=\bigcup\{cl_{\tau}(\{x\}): x\in A\}.$
		
		Then there exist $a\in A$ and $y\in cl_{\tau}(\{a\}) - A$, thus  $y \thicksim_{\tau} a$ and so $y, a\in [x_i]_{\thicksim_{\tau}}$ for some $i\in I$.
		
		Choose an $f\in P$ such that $y=f(i)$.  Then  $cl_{\tau_f}(\{a\})=\{a, f(i)\}=\{a, y\}$.  It follows that $A\not\in \Gamma(X, \tau_f)$, otherwise $y\in cl_{\tau_f}(\{a\})\subseteq cl_{\tau_f}(A)=A$, contradicting $y\not\in A$.

Note that  $\tau_{f}$ is $T_0$ and  $\tau\subseteq \tau_f$, thus $\lambda \subseteq \tau_{f}$.  Hence, $A\not\in \Gamma(X, \lambda)$.

\vskip 0.2cm
		Case 2: $A=\tilde{A}.$
		
		Then for any $x\in A$, $y\thicksim_{\tau}x$ implies $y\in A$.
		
		Since $A\not\in \Gamma(X, \tau)$, there exists  $b\in cl_{\tau}(A) - A$.
		
		Choose any $f\in P$ and let $M =F_{f}(X)=\{f(x): x\in X\}$. Take $(M, \nu)$ be the subspace of $(X, \tau)$ with underlying set $M$.
		Then, as different element of $M$ are not $\thicksim_{\tau}$ equivalent, $(M, \nu)$ is $T_0$, thus $(X, \tau_{M})$ is $T_0$ by
		Lemma \ref{topology by subspace}.
		
		For any $F\in \Gamma(X, \tau)$, $F=(F-M)\cup (F\cap M)$, and $F-M\subseteq X-M, F\cap M\in \Gamma(M, \nu)$, hence,
		$F\in\Gamma(X, \tau_{M})$. It follows that $\tau\subseteq \tau_{M}$.
		
		We now show that $A\not\in \Gamma(X, \tau_M)$. Assume, on the contrary, that  $A\in \Gamma(X, \tau_M)$. Then
		$A=C\cup (E\cap M)$, where $C\cap M=\emptyset$ and $E\in \Gamma(X, \tau)$.
		
		If $F_f(b)\in E\cap M\subseteq A$, then $b\in cl_{\tau}(\{F_{f}(b)\})\in \tilde{A}=A$,
which  contradicts  $b\in cl_{\tau}(A) - A$. Hence, $F_{f}(b)\not\in E\cap M$.
		
		Also note that $F_f(b)=f(i)\in M$, where $b\thicksim_{\tau} f(i)$.   Thus $F_{f}(b)\in  X-E \in \tau$. Since $F_f(b)\in cl_{\tau}(\{F_{f}(b)\})= cl_{\tau}(\{b\})$, then $b\in X-E$.  Hence, $A\cap (X-E)\not=\emptyset$ because $b\in cl_{\tau}(A)\cap (X-E)$ and $X-E\in \tau$.

	Choose  $a\in A\cap (X-E)$. Then, as $cl_{\tau}(\{a\})= cl_{\tau}(\{F_f(a)\})$, we have that $F_f(a)\in X-E$ and $F_{f}(a)\in A$. However, $F_f(a)\in M$ and $F_f(a)\in A =C\cup (E\cap M)$ (note that $C\subseteq X-M$). So $F_{f}(a)\in E\cap M$, implying  $F_{f}(a)\in E$. This contradicts  $F_f(a)\in X-E$.
		
		Hence $A\not\in \Gamma(X, \tau_{M})$.
		
		In summary, for any $A\not\in \Gamma(X, \tau)$, there is a $T_0$ topology $\mu \in \mathrm{T}(X)$ such that $\tau\subseteq \mu$ and $A\not\in\Gamma(X, \mu)$.
		
		Hence, $\lambda=\bigcap\{\mu\in \mathrm{T}(X): \tau\subseteq \mu,~\mu \mbox{ is } T_0\}\subseteq \tau.$
		At last, we have
$$\bigcap\{\mu\in \mathrm{T}(X): \tau\subseteq \mu,~\mu \mbox{ is } T_0\}=\tau.$$
	\end{proof}
	
	\begin{lemma}\label{topology tau}
		Let $(X, \tau)$ be a $T_0$ space. Then there is a collection $\mathcal{A}$ of sober topologies on $X$ such that
		$$\tau=\bigcap \mathcal{A}.$$
	\end{lemma}
	\begin{proof}
	It is enough to prove that $\nu=\tau$, where  $\nu=\bigcap\{\mu\in T_{sob}(X): \tau\subseteq \mu\}$.
		
	As the discrete topology on $X$ is sober and contains $\tau$,  the family $ \{\lambda\in T_{sob}(X): \tau\subseteq \lambda\}$ is nonempty and clearly $\tau\subseteq \nu$.
		
  It remains to show that if $A\not\in \Gamma(X, \tau)$, then there is a sober topology $\mu$  containing $\tau$ and $A\not\in\Gamma(X, \mu)$.
		
		Case 1: $A\not=\tilde{A}.$
		
		Then there exists $b\in  \bigcup\{cl_{\tau}(\{x\}): x\in A\} - A$, that is, $b\in cl_{\tau}(\{a\}) - A$ for some $a\in A$.
		Let $M=\{a, b\}$ with the subspace topology $\tau_{a, b}$, which equals $\{\emptyset, \{a, b\}, \{a\}\}$. Then clearly $(M, \tau_{a, b})$  is sober (it is homeomorphic to the Sierpinski space).
		By Lemma \ref{topology by subspace}, $(X, \tau_{M})$ is sober. Also, similar to the conclusion in the case 2 of  the proof of  Lemma \ref{every topology is the meets of t0 topo}, $\tau\subseteq \tau_M$.  If $A\in \Gamma(X, \tau_{M})$, then $A=C\cup (D\cap M)$ where $C\subseteq X-M, D\in\Gamma(X, \tau)$. As $a\in M$,  $a\not\in C$, we have $a\in D$. Then $b\in cl_{\tau}(\{a\})\subseteq D\subseteq A$, which contradicts the assumption that $b\not\in A$.
		
		Hence $A\not\in \Gamma (X, \tau_M)$.
		\vskip 0.2cm

		Case 2: $A=\tilde{A}.$
		
		Since $A\not=cl_{\tau}(A)$, there exists $b\in cl_{\tau}(A) - A$.  Let $M=A\cup \{b\}$ and $\lambda_{M}$ be the topology on $M$ such that $\{(C\cap M )-\tilde{B}: C\in \Gamma(X, \tau), B\subseteq A\}$ is a subbase of  $\Gamma(M, \lambda_{M})$.
		
		(i) $(M, \lambda_{M})$ is sober.
		
		For any $z\in M$, as $(X, \tau)$ is $T_0$, $\{z\}=cl_{\tau}(\{z\}) - \tilde{B_z}=(cl_{\tau}(\{z\})\cap M) - \tilde{B_z} $, where $B_z=cl_{\tau}(\{z\}) -\{z\}$. In addition  $b\not\in B_z$ (if $z=b$, then clearly $b\not\in B_z$; if $b\not=z$, then $z\in A$, so $B_z\subseteq cl_{\tau}(\{z\})\subseteq \tilde{A}=A$, then $b\not\in B_z$ as $b\not\in A$).
		
		By the definition of $\lambda_{M}$, if $z\not=b$, then $z\in A$,  $(cl_{\tau}(\{z\})\cap M) - \tilde{B_z}\in\Gamma(M, \lambda_{M})$, hence $\{z\}$ is closed in $(M, \lambda_{M})$. For $b\in M$, $\{b\}=(cl_{\tau}(\{b\})\cap M) - A=(cl_{\tau}(\{b\})\cap M) - \tilde{ A}$ is also closed in $(M, \lambda_M)$. It follows that $(M, \lambda_{M})$ is $T_1$.
		
		We now show that $(M, \lambda_{M})$ is sober. Let $D\in \Gamma(M, \lambda_M)$ and $D$ contains two distinct elements $a, c$.
        If one of $a$ and $c$ equals $b$, say  $a=b$, then $c\in A$ and $a\not\in cl_{\tau}(\{c\})$. Otherwise both $a$ and $c$ are in $A$ and  we can assume that $a\not\in cl_{\tau}(c)$.

        Let
		$$D =\bigcap_{j\in J}( \bigcup_{k\in K(j)}(C_k-\tilde{B_k})),$$
		where $C_{k}$ is a closed subset in the subspace $M$ of $(X, \tau)$, $K(j)$ is a finite set and $B_k\subseteq A$.

		Then
		$$D \subseteq (D - cl_{\tau}(\{c\}))\cup (cl_{\tau}(\{c\})\cap M).$$
		Also $D - cl_{\tau}(\{c\})=\bigcap_{j\in J}( \bigcup_{k\in K(j)}(C_k-\tilde{(B_k\cup \{c\})}))\in\Gamma(M, \lambda_M)$ (note that $B_k\cup \{c\}\subseteq A$ as we assume $c\in A$) and
		$cl_{\tau}(\{c\})\cap M\in \Gamma(M, \lambda_M)$. But $D\not\subseteq D - cl_{\tau}(\{c\})$ and $D\not\subseteq cl_{\tau}(\{c\})\cap M $ (note $a\in D$ and $a\not\in cl_{\tau}(\{c\}))$. Thus $D$ is not irreducible.
		So every irreducible closed set of $(M, \tau_M)$ is a singleton, hence $(M, \lambda_M)$  is sober.

By Lemma \ref{topology by subspace}, the space $(X, \tau_M)$ is also sober.
		
		For any $F\in \Gamma(X, \tau)$, $F=(F\cap M)\cup (F-M) \in \Gamma(X, \tau_{M})$, thus $\Gamma(X, \tau)\subseteq \Gamma(X, \tau_M)$.

		We now show that $A\not\in \Gamma(X, \tau_M)$.
		
		For any  $D\in \Gamma(X, \tau_M)$,  $D=E\cup F$ for some  $E\in\Gamma(M, \tau_M)$ and $F\subseteq X-M$. We show that $A\not=D$ by considering the following cases:
		
		(i)  $F\not=\emptyset$.
		
		Then $D\not\subseteq M$, so $D\not=A$ because $A\subseteq M$.
		
		(ii) $F=\emptyset$ and $b\in D$.
		
		Then, as $b\not\in A$,  $D\not=A$.
		
		(iii) $F=\emptyset$ and $b\not\in D$.
		
		Assume that $D =\bigcap_{j\in J}( \bigcup_{k\in K(j)}(C_k-\tilde{B_k}))$, where $C_{k}$ is a closed subset in the subspace $M$ of $(X, \tau)$, $B_k \subseteq A$ (thus $\tilde{B_k}\subseteq A$)  and $K(j)$ is a finite set.
		
		Then, as $b\not\in D$, there is a $j_0\in J$ such that $b\not\in \bigcup_{k\in K(j_0)}(C_k - \tilde{B_k})$.

        Since  $b\not\in \tilde{B_k}$ for each $k$, $b\not\in \bigcup_{k\in K(j_0)}C_k$, thus $\bigcup_{k\in K(j_0)}C_k\subseteq A$.

		We now show that  $\bigcup_{k\in K(j_0)}C_k=A$ does not hold. If not, let $C_k= G_k\cap M$ with $G_k\in \Gamma(X, \tau)$ for each $k\in K(j_0)$.
		Then $X-\bigcup_{k\in K(j_0)}G_k\in \tau$ and $b\in X-\bigcup_{k\in K(j_0)}G_k$. This contradicts $b\in cl_{\tau}(A)$ and
		$A\cap (X-\bigcup_{k\in K(j_0)}G_k) \subseteq A\cap (X-  \bigcup_{k\in K(j_0)}C_k)=A\cap (X - A)=\emptyset$.
        Therefore,  $\bigcup_{k\in K(j_0)}C_k\not=A$. Then, as $D\subseteq \bigcup_{k\in K(j_0)}C_k\subsetneq A$, we have  $D\not=A$.
		\vskip 0.2cm

		In summary, for any $A\not\in \Gamma(X, \tau)$, there is a sober topology $\mu\in \mathrm{T}(X)$ such that $\tau\subseteq \mu$ and $A\not\in\Gamma(X, \mu)$.
		
		Hence, $\tau=\bigcap\{\mu\in T_{sob}(X): \tau\subseteq \mu\}$, as desired.
		The proof is completed.
	\end{proof}
	
	The combination of Lemma \ref{every topology is the meets of t0 topo} and Lemma \ref{topology tau} deduces the main result in this section.
	
	\begin{theorem}\label{sob meet dense}
		Every topology on a set $X$ is the meet of some sober topologies on $X$.
	\end{theorem}
	
\vskip 0.3cm
	Thus $\mathrm{T}_{sob}(X)$ is meet dense in $\mathrm{T}(X)$.
	
	The following result follows from Theorem \ref{sob meet dense} straightforwardly.
	
	\begin{corollary}
		A topology $\tau\in \mathrm{T}(X)$ is $T_1$ if and only if it is  the meet of some $T_1$ sober topologies.
	\end{corollary}

Proposition \ref{directed join of sober topologies} says that the join of a directed family of sober topologies on a set is also sober. It is then nature to ask whether the meet of a filtered family of  sober topologies is sober. The following example shows that the  answer is no, even each topology is $T_2$.

To verify the counterexample below, we need to make use of  the Chinese Remainder Theorem shown below.
Let $\mathbb{N}_0$ be the set of all nonnegative integers and $\mathbb{N}$ the set of all positive integers. For each $n, x\in \mathbb{N}_0$, let
$$[x]_{n}=\{y\in \mathbb{N}_0: y\equiv x (mod~ n) \}=\{kn+x: k\in\mathbb{N}_0\}.$$

\begin{lemma} (Chinese Remainder Theorem)
Let $n_1, n_2, \cdots, n_k$ be pairwise coprime integers with $n_i > 1$ for each $i$. Assume that for each $1 \le i \le k$, $a_i$ is an integer such that $0\le a_i < n_i$. Put $N =n_1n_2\cdots n_k$ be the product of $n_i's$. Then there is a unique integer $x$ such that $0\le x < N$ and $x \equiv a_i (mod~ n_i)$ for each $ 1\le i\le k$.
\end{lemma}
\vskip 0.5cm
In particular, if $\{p_1, p_2, \cdots, p_m\}$ is a finite set of distinct prime numbers and for each $ 1\le i\le m$, $b_i$ is an integer such that $0\le b_i <p_i$, then
$$ \bigcap\{[b_i]_{p_i}: 1 \le i\le m\}\not=\emptyset.$$

\begin{example}\label{meet of chain of t2 topology not sober}
Let $X=\mathbb{N}_0$  and
$P=\{p_k: k\in\mathbb{N}\}$ be the set of all prime numbers, where we assume that $p_k < p_{k+1}$ for all $k$.

For each $n\in \mathbb{N}$, let $\tau_n$ be the topology on $X$  of which
$\mathcal{S}_{n}=\{[k]_{p_t}:  n\le t,  0\le k< p_t\}$ is a subbase.

Clearly  $\{\tau_n: n\in \mathbb{N}\}$ is a decreasing chain of  topologies on $\mathbb{N}_{0}$. We now verify that each $\tau_n$ is $T_2$ and their meet is not sober.

(1) For each  $n\in \mathbb{N}$, $\tau_n$ is $T_2$.

As a matter of fact, for any  $a, b\in X$ with $a\not= b$, there is a prime number $p_m$ such that  $n\le m$ and $a + b < p_m$.
Then $a\in U= [a]_{p_m}\in \tau_n, b\in V= [b]_{p_m}\in\tau_n$.
Also $U\cap V=[a]_{p_m}\cap [b]_{p_m}=\emptyset$ (because  $a\not=b$ and $a, b < p_m$).

(2) Let $\tau=\bigwedge \{\tau_n: n\in\mathbb{N}\}$. We show that $X$ is irreducible under $\tau$.
Assume that $U_1\in\tau, U_2\in\tau$ are nonempty open sets in $\tau$ (we show that $U_1\cap U_2\not=\emptyset$).

Choose $n_1\in U_1$ and $n_2\in U_2$. We just consider the nontrivial case $n_1\not=n_2$.

Note that if $p\in P$ and $[k]_{p}\cap [h]_{p}\not=\emptyset$ and $0\le k <p, 0\le h < p$, then $k=h$.

Since $U_1\in \tau\subseteq \tau_1$ and $U_1\not=\emptyset$, it follows that there is a finite set $F_1\subseteq \mathbb{N}$ such that
$n_1\in \bigcap \{[h_{p_i}]_{p_i}: i\in F_1\}\subseteq U_1$, where $0\le h_{p_i} < p_{i}$ for each $i\in F_1$.  By the definition of elements in $P$,  $p_{i}'s  (i\in F_1)$ are different primes.

Let $m=\max\{i: i\in F_1\}$ be the largest member of $F_1$.
 Again, since $n_2\in U_2\in\tau\subseteq \tau_{m+1}$, there is  a finite set $F_2\subseteq \mathbb{N}$ such that
 $n_2\in \bigcap\{[g_{p_{j}}]_{p_{j}}: j\in F_2\}$, where $0\le g_{p_{j}} < p_{j}$ and $j \ge m+1$ for each $j\in F_2$.

Now $F_1\cap F_2=\emptyset$, so $\{p_j: j\in F_1\cup F_2\}$ is a set of distinct primes. By the Chinese Remainder Theorem,
$$ \bigcap \{[h_{p_i}]_{p_i}: i\in F_1\}\cap \bigcap \{[g_{p_j}]_{p_j}: j\in F_2\}\not=\emptyset.$$
As $\bigcap \{[h_{p_i}]_{p_i}: i\in F_1\}\cap \bigcap \{[g_{p_j}]_{p_j}: j\in F_2\}\subseteq U_1\cap U_2$, we have  $U_1\cap U_2\not=\emptyset$. It follows that  $X$ is an irreducible closed set in $(X, \tau)$.

As the intersection of $T_1$ topologies on a set is always $T_1$, $(X, \tau)$ is $T_1$. Since $X$ is an infinite set, $X$ is not the closure of any  singleton set. Therefore $(X, \tau)$ is not sober.

\end{example}
	
Since every $T_2$ topology is sober, the above example also shows that the meet of a decreasing countable chain of $T_2$ topologies need not be $T_2$.

\begin{remark}
It is well known that the  co-finite topology $\tau_{cof}$ ($U\in\tau_{cof}$ if either $U=\emptyset$  or $X-U$ is a finite set) on each non-singleton set  $X$ is
the smallest  $T_1$ topology on $X$ and is not sober.

It is natural to wonder whether the topology $\tau$ in Example \ref{meet of chain of t2 topology not sober} is actually  the co-finite topology on $\mathbb{N}_0$. We now show that it is not.

Let $A=\{0, p_1, p_1 p_2, \cdots, \}=\{0\}\cup\{\prod^n_{i=1}p_i: n\in\mathbb{N}\}$. We will show that $ X- A\in\tau_n$ for each $n\in\mathbb{N}$, thus $X - A\in\tau$. Hence, $\tau$ is different from the $\tau_{cof}$.

	Fix $n\in\mathbb{N}$ and given $x\in(X-A)$.
	
	Then there is a prime number $p_m$ such that $n < m$, $x < p_m$. Thus $x\in\bigcap{\{[x]_{p_i}: m\le i\le 2m\}}\in \tau_n$. And we only need to prove that $\bigcap{\{[x]_{p_i}: m\le i\le 2m\}}\subseteq (X-A)$ (i.e., $\bigcap{\{[x]_{p_i}: m\le i\le 2m\}}\cap A=\emptyset$).
	
	Let $A_1=\{\prod^n_{i=0}p_i: 1\leq n\leq m\}$ and $A_2=\{\prod^n_{i=1}p_i: m<n\}$. Then $A=A_1\cup A_2$.
	
	First, we will prove that $A_1\cap \bigcap{\{[x]_{p_i}: m\le i\le 2m\}}=\emptyset$. Clearly, $x\not\in A_1$ because $x\not\in A$. Note that $\max(A_1)=\prod^m_{i=1}p_i$ and
	$\min(\bigcap{\{[x]_{p_i}: m\le i\le 2m\}}-\{x\})=x+\prod^{2m}_{i=m}p_i$. Hence $\min(\bigcap{\{[x]_{p_i}: m\le i\le 2m\}}-\{x\})>\max(A_1)$, and it implies $\bigcap{\{[x]_{p_i}: m\le i\le 2m\}}\cap A_1=\emptyset$.
	
	Second, we will prove that $A_2\cap \bigcap{\{[x]_{p_i}: m\le i\le 2m\}}=\emptyset$. Due to $x\not\in A$, $x\not=0$. Because $\bigcap{\{[x]_{p_i}: m\le i\le 2m\}}\subseteq [x]_{p_m}$, $A_2\subseteq [0]_{p_m}$ and $[x]_{p_m}\cap [0]_{p_m}=\emptyset$, we have $\bigcap{\{[x]_{p_i}: m\le i\le 2m\}}\cap A_2=\emptyset$.
	
	Thus $\bigcap{\{[x]_{p_i}: m\le i\le 2m\}}\cap A=\emptyset$. It follows that $(X-A)\in\tau_n$ for each $n\in\mathbb{N}$. That is $A$ is closed in $(X,\bigwedge_{n\in\mathbb{N}}\tau_n)$.
	
	Actually, $A$ is also compact. Given an open cover $\{U_j: j\in J\}$ of $A$. Then $0\in U_{j_0}$ for some $j_0\in J$. Due to $U_{j_0}\in\tau_1$, there is a finite set $F\subseteq \mathbb{N}$ with $a\ge 1$ for each $a\in F$, such that $0\in \bigcap\{[0]_{p_i}: i\in F\}\subseteq U_{j_0}$. It follows that $\{\prod^n_{i=1}p_i:n>\max(F)\}\subseteq U_{j_0}$. And for each $1\le n\le \max(F)$, there is $j_n\in J$ such that $\prod_{i=1}^{n}p_{i}\in U_{j_n}$. Thus $\{U_{j_n}:0\le n\le \max(F)\}$ is a finite subcover of $A$.
\end{remark}

A space  $(X, \tau)$  is  Alexandroff-discrete  if the intersection of any collection of open sets is open. In this case, we call $\tau$ an Alexanderoff - discrete topology.

It then follows easily that a meet of any collection of  Alexandroff-discrete topologies on a set $X$ is also an  Alexandroff -discrete  topology.

Given any poset $(P, \le)$, let $\Upsilon(P, \le)$ (or just $\Upsilon(P)$) be the set of all upper sets of $P$. Then
$(P, \Upsilon(P, \le))$ is a $T_0$ Alexanderoff - discrete space. The topology $\Upsilon(P, \le)$ is called the Alexanderoff topology on $P$\cite{Goubault-2013}\cite{jonstone-82}.

\begin{remark}\label{alexanderoff topology 3}
\II
\I[(1)] A space $(X, \tau)$ is  $T_0$ Alexanderoff-discrete if and only if there is a partial order $\le$ on $X$ such that $\tau=\Upsilon(P, \le)$ (see Exercise 4.2.13 of \cite{Goubault-2013}).
\I[(2)] A subset $F$ of poset $P$ is a closed irreducible set of $(P, \Upsilon(P, \le))$ if and only if $F$ is a directed and lower set of $(P, \le)$.
\I[(3)] For any $x\in P$, the closure of $\{x\}$ in $(P, \Upsilon(P, \le))$ equals $\downarrow x=\{y\in P: y\le x\}$.
Now $\downarrow x -\{x\}$ is still a lower set hence a closed set in $(P, \Upsilon(P, \le))$. Hence $(P, \Upsilon(P, \le))$ is a $T_D$ space.
\I[(4)] If an Alexanderoff - discrete topology $\tau$ is sober and $\mu$ is a topology on the same set such that $\tau\subseteq \mu$,
then as $\tau$ is also $T_D$, by Proposition \ref{sober td spaces}(2), $\mu$ is also sober.
\III
\end{remark}

\begin{proposition}\label{alexanderoff}
 Every Alexanderoff-discrete topology is the meet of  some sober Alexanderoff - discrete topologies.
\end{proposition}
 \begin{proof}
  Let $(X, \tau)$ be an Alexanderoff-discrete space. We can assume that
 $\tau=\Upsilon(X, \le)$, where $\le$ is a partial order on $X$.

 For each pair of elements  $a, b$ in $X$ with $a\le b$ and $a\not=b$, let
 $$\le_{(a, b)}=\{(x, x): x\in X\}\cup \{(a, b)\}.$$
 Then $\le_{(a, b)}$ is a partial order on $X$. A subset $F\subseteq X$ is a closed irreducible set in $(X, \Upsilon(X, \le_{(a, b)}))$ if and only if either $F=\{x\}$ for some $x\in X$ with $x\not=b$, or $F=\{a, b\}$. In the first case, $F=cl(\{x\})$, while in the second case $F=cl(\{b\})$. Thus $(X, \Upsilon(X, \le_{(a, b)}))$ is sober.

 We now show that
 $$\tau=\bigwedge\{\Upsilon(X, \le_{(a, b)}): a\not=b, a\le b\}.$$
Clearly,  $\tau\subseteq \bigwedge\{\Upsilon(X, \le_{(a, b)}): a\not=b, a\le b\}$.
Now let $U\in \bigwedge\{\Upsilon(X, \le_{(a, b)}): a\not=b, a\le b\}$. Let $x\le y$ and $x\in U$. If $y=x$, then immediately, $y\in U$. If $y\not=x$, then $x\le_{(x, y)} y$. Since $U\in \Upsilon(X, \le_{(x, y)})$ and $x\in U$, thus $y\in U$. It follows that
$U\in \Upsilon(X, \le)=\tau$.
It follows that $\tau=\bigwedge\{\Upsilon(X, \le_{(a, b)}): a\not=b, a\le b\}.$
The proof is completed.
\end{proof}

\section{Minimal sober topologies}

People have studied minimal members of several classes of topologies, such as minimal Hausdorff topologies, minimal regular topologies, minimal locally compact topologies and minimal normal topologies (see \cite{berri-1962} \cite{cameron-1971}).
In this section, we study  minimal sober topologies. One necessary and sufficient conditions of such topologies is obtained.

A topology $\tau$ on $X$ is called a {\sl minimal sober topology}, if it is sober and there is no sober topology on $X$ strictly coarser than $\tau$.

A topological space $(X, \tau)$ is called a {\sl minimal sober space}, if $\tau$ is a minimal sober topology on $X$.

We first propose a construction method, that yields a necessary condition for minimal sober spaces.

\begin{lemma}\label{construction}
	Let $(X,\tau)$ be a $T_0$ space and $x, y\in X$ are noncomparable elements with respect to the specialization order $\le_{\tau}$ (that is, $x\not\in cl_{\tau}(\{y\})$ and $y\not\in cl_{\tau}(\{ x\})$).
	
	Define $\tau^*=\{U\in\tau: x\in U\}\cup\{U\in\tau: x\not\in U, y\not\in U\}.$
	
	If $(X,\tau)$ is a $T_0$ space, then $(X,\tau^*)$ is also a $T_0$ space.
\end{lemma}
\begin{proof}
	
	It is easy to see that $\tau^{*}$ is a topology and $\tau^{*} \subseteq \tau$. The set of all closed sets of $(X,\tau^*)$ is  $$\Gamma(X,\tau^*)=\{C\in\Gamma(X,\tau): x\not\in C\}\cup \{C\in\Gamma(X), \{x, y\}\subseteq C\}.$$
	
	Claim 1. For each $a\in X$, it holds that\[
	cl_{\tau^*}(\{a\}) =	
\begin{cases}
		cl_{\tau}(\{a\}), & \quad x\not\in cl_{\tau}(\{a\}), \\
		cl_{\tau}(\{a\})\cup cl_{\tau}(\{y\}), & \quad x\in cl_{\tau}(\{a\}).
	\end{cases}
	\]
	
	As a mater of fact, if $x\not\in cl_{\tau}(\{a\})$, then  $cl_{\tau}(\{a\})\in \Gamma(X,\tau^*)$ and $cl_{\tau}(\{a\})\subseteq cl_{\tau^{*}}(\{a\})$,  hence  $cl_{\tau^{*}}(\{a\})=cl_{\tau}(\{a\})$.
	
	If $x\in cl_{\tau}(\{a\})$, then $a\in cl_{\tau}(\{a\}) \cup cl_{\tau}(\{y\})\in \Gamma(X,\tau^*)$, showing that $cl_{\tau^*}(\{a\})\subseteq cl_{\tau}(\{a\}) \cup cl_{\tau}(\{y\})$.  In addition, $cl_{\tau^*}(\{a\})\in \Gamma(X,\tau^*)$ and
$x\in cl_{\tau}(\{a\})\subseteq cl_{\tau^*}(\{a\})$, so $x\in cl_{\tau^*}(\{a\})$. Hence $y\in cl_{\tau^*}(\{a\})$. Therefore,
$cl_{\tau}(\{a\})\cup cl_{\tau}(\{y\})\subseteq cl_{\tau^*}(\{a\})$. All these together show that
$cl_{\tau^{*}}(\{a\})=cl_{\tau}(\{a\})\cup cl_{\tau}(\{y\})$.
	
Hence, $\le_{\tau^*}=\le_\tau\cup\{(u,v):u\le_\tau y \mbox{ and }  x\le_\tau v\}$.
	
	\vskip 0.5cm
	
	Claim 2. $(X, \tau^{*})$ is $T_0$.
	\vskip 0.2cm
	Let $a, b\in X$ with $a\not=b$. Since $(X, \tau)$ is $T_0$, without lose of generality, we can assume that $a\not\in cl_{\tau}(\{b\})$.
	
	(i) If $x\not\in cl_{\tau} (\{b\})$,  by claim 1, $cl_{\tau}(\{b\})=cl_{\tau^{*}}(\{b\})$.  It implies that $a\not\in cl_{\tau}(\{b\})=cl_{\tau^*}(\{b\})$.
	
	(ii) Now assume that $x\in cl_{\tau}(\{b\})$. Then by claim 1, $cl_{\tau^{*}}(\{b\})=cl_{\tau}(\{b\})\cup cl_{\tau}(\{y\})$.
	
	If $a\not\in cl_{\tau}(\{y\})$, then $a\not\in cl_{\tau}(\{b\})\cup cl_{\tau}(\{y\})=cl_{\tau^{*}}(\{b\}).$
	
	If $a\in cl_{\tau}(\{y\})$, then $x\not\in cl_{\tau}(\{a\})$ because $x\not\in cl_{\tau}(\{y\})$. By claim 1,  $cl_{\tau^{*}}(\{a\})=cl_{\tau}(\{a\})$. Now as $x\in cl_{\tau}(\{b\})$ and $x\not\in cl_{\tau}(\{a\})$, we have that
	$b\not\in cl_{\tau}(\{a\})=cl_{\tau^{*}}(\{a\})$.
	
	All these together show that $(X, \tau^{*})$ is $T_0$.
	\end{proof}
	
	\begin{lemma}\label{n of sober}
		Let  $\tau$ be a topology on $X$ and $\tau^{*}$ be the topology defined from $\tau$ and non-comparable elements $x, y\in X$ as in  \rm{Lemma \ref{construction}}.
\II
\I[(1)] If $(X,\tau)$ is a sober space,
		then $(X,\tau^*)$ is  a sober space.
\I[(2)] If $(X,\tau)$ is a well-filtered space, then  $(X,\tau^*)$ is  a well-filtered space.
\I[(3)] If $(X,\tau)$ is a d-space, then  $(X,\tau^*)$ is  a d-space.
\III
	\end{lemma}
	\begin{proof} First, note that for any $H\in \Gamma(X, \tau)$, $H\cup cl_{\tau}(\{y\})\in \Gamma(X, \tau^*)$ (if $x\in H$, then $\{x, y\}\subseteq H\cup cl_{\tau}(\{y\})$; if $x\not\in H$, then $x\not\in H\cup cl_{\tau}(\{y\})$).
\vskip 0.5cm

  (1)  Let $(X,\tau)$ be sober and  $C\in Irr_c(X,\tau^*)$.
	
	(i) Assume $x\not\in C$.
	
	If $C= C_1\cup C_2$ for some $C_1,C_2\in\Gamma(X,\tau)$, then $C_1,C_2\in\Gamma(X,\tau^*)$ by $x\not\in C$.
	Thus $C=C_1$ or $C=C_2$, as $C\in Irr_c(X,\tau^*)$. It follows that $C\in Irr_c(X,\tau)$. Since $(X, \tau)$ is sober, $C=cl_{\tau}(\{a\})$ for some $a\in X$. As $x\not\in C$, so $x\not\in cl_{\tau}(\{a\})$. By Lemma \ref{construction}, $cl_{\tau^{*}}(\{a\})=cl_{\tau}(\{a\})$. Hence $C=cl_{\tau^{*}}(\{a\})$.
	
	(ii) Let $x\in C$. Then $cl_{\tau^*}(\{x\})=cl_{\tau}(\{x\})\cup cl_{\tau}(\{y\})$, which is contained in $C$ as $C\in \Gamma(X, \tau^*)$. Hence $cl_{\tau}(\{y\})\subseteq C$.
	
	Let $\mathcal{A}=\{F:F\in\Gamma(X,\tau),C=F\cup cl_{\tau}(\{y\})\}$. By the above deduced fact,  $C=C\cup cl_{\tau}(\{y\})$, hence $C\in \mathcal{A}$.

 It follows that $C=(\bigcap_{F\in\mathcal{A}}F)\cup cl_{\tau}(\{y\})$.
 	
 Let $\widehat{F}=\bigcap_{F\in\mathcal{A}}F$. Then $C=\widehat{F}\cup cl_{\tau}(\{y\})$.

 Assume $\widehat{F}= D\cup E$ for some $D,E\in\Gamma(X,\tau)$.

Then $C= (D\cup cl_\tau(\{y\}))\cup (E\cup cl_\tau(\{y\}))$. Since $C$ is irreducible in $(X, \tau^*)$, hence
$C= D\cup cl_\tau(\{y\})$ or $C= E\cup cl_\tau(\{y\})$, which further deduces that $\widehat{F}=D$ or $\widehat{F}=E$ (note that if  $C= D\cup cl_\tau(\{y\})$, for example, then $D\in \mathcal{A}$, thus $D\subseteq \widehat{F}\subseteq D$, implying $D=\widehat{F}$).

That is $\widehat{F}\in Irr_c(X,\tau)$. As $(X, \tau)$ is sober, $\widehat{F}=cl_{\tau}(\{a\})$ for some $a\in X$. By $x\in C=cl_{\tau}(\{a\})\cup cl_{\tau}(\{y\})$ and $x\not\in cl_{\tau}(\{y\})$, we have $x\in cl_{\tau}(\{a\})$.
 Thus $$C=\widehat{F}\cup cl_{\tau}(\{y\})=cl_{\tau}(\{a\})\cup cl_{\tau}(\{y\})=cl_{\tau^*}(\{a\}).$$

Since $(X, \tau^*)$ is  $T_0$, it is sober.

\vskip 0.5cm

 (2) 	Let $(X, \tau)$ be well-filtered. Assume that $K$ is a compact saturated set of $(X,\tau^*)$ and $\{U_i\in\tau:i\in I\}\subseteq \tau$ be an open cover of $K$.
	
	(i) Assume  $x\not\in K$.

Note that $cl_{\tau^*}(\{x\})=cl_\tau(\{x\})\cup cl_\tau(\{y\})$. As $K$ is the intersection of all open sets $V\in \tau^{*}$ containing $K$, one has that $cl_{\tau^*}(\{x\})\cap K=\emptyset$. Now
$K=K-cl_{\tau^*}(\{x\})\subseteq \bigcup\{U_i - cl_{\tau^*}(\{x\}):i\in I\}
= \bigcup\{U_i - (cl_\tau(\{x\})\cup cl_\tau(\{y\})):i\in I\}$. Not that each
$U_i - (cl_\tau(\{x\})\cup cl_\tau(\{y\}))\in \tau^*$.
As $K$ is compact in $(X, \tau^{*})$,  $K\subseteq \bigcup\{U_i - (cl_\tau(\{x\})\cup cl_\tau(\{y\})): i\in F\} \subseteq \bigcup\{U_i:i\in F\}$ for some finite $F\subseteq I$.

 (ii) 	Assume $x\in K$.

 Then $x\in U_{i_0}$ for some $i_0\in I$ and $K\subseteq \bigcup\{U_{i_0}\cup U_i: i\in I\}$. Because $U_{i_0}\cup U_i\in \tau^*$ for each $i\in I$, $K\subseteq \{U_{i_0}\cup U_i:i\in F\}$ for some finite $F\subseteq I$.

 Thus  $K$ is also a compact set of $(X,\tau)$. Furthermore, $K$ is saturated in $(X, \tau)$ by $\tau^*\subseteq \tau$.
	
Then by the definition of well-filtered spaces and  that $\tau^*\subseteq \tau$,  $(X,\tau^*)$ is also well filtered.
	
	\vskip 0.5cm
	
 (3)  Assume that  $(X,\tau)$ is a d-space.

 We show that $(X,\tau^*)$ is also a d-space.
	\vskip 0.2cm
	Let $U\in\tau^*$ and $D$ be a directed set of $(X,\le_{\tau^*})$ such that $\bigcap\{\uparrow_{\tau^*}d: d\in D\}\subseteq U$.
	As $\tau^*\subseteq \tau$, $D$ is also a directed subset of $(X, \le_{\tau})$.

	(i) Assume $y\in \uparrow_{\tau^*}d$ for each $d\in D$.
	
Now  $\uparrow_{\tau^*}d=\{a\in X: d\in cl_{\tau^*}(\{a\})\}=\{a\in X: d\in cl_{\tau^*}(\{a\}), x\not\in cl_{\tau}(\{a\})\}\cup \{a\in X: d\in cl_{\tau^*}(\{a\}), x\in cl_{\tau}(\{a\})\}$
From this, we can deduce easily that
$\uparrow_{\tau^*}d=\uparrow_\tau d \cup \uparrow_\tau x$.

 It follows that $\bigcap\{\uparrow_{\tau^*}d:d\in D\}=\bigcap\{\uparrow_{\tau}d\cup\uparrow_\tau x:d\in D\}=\bigcap\{\uparrow_{\tau}d:d\in D\}\cup \uparrow x\subseteq U$.
 Thus $\bigcap\{\uparrow_{\tau}d:d\in D\}$, $D$ is a directed set in $(X, \le_{\tau})$ and $U\subset \tau$. As $(X, \tau)$ is a d-space,  $\uparrow_\tau d_{0} \subseteq U$ holds for some $d_0\in D$.
Then $\uparrow_{\tau^*} d_0 =\uparrow_{\tau}d_0 \cup\uparrow_\tau x  \subseteq U$.
	
	(ii) Now assume $y\not\in \uparrow_{\tau^*}d_0$ for some $d_0\in D$.

   Then $y\not\in \uparrow_{\tau^*}\!d $ for all $d\in D$ with $d_0\le_{\tau^*}d$, that is $d\not\in cl_{\tau^*}(\{y\})=cl_{\tau}(\{y\})$.  For each such $d$,
   $d\le_{\tau^*} a$ if and only if  $d\in cl_{\tau^*}(\{a\})$,  if and only if $d\in cl_{\tau}(\{a\})$, if and only if  $a\in \uparrow_{\tau} d$.
   	
  Hence $\bigcap\{\uparrow_{\tau^*}d:d\in D\}=\bigcap\{\uparrow_{\tau^*}d:d\in D,d_0\le_{\tau^*}d\}=\bigcap\{\uparrow_{\tau}d:d\in D,d_0\le_{\tau}d\}\subseteq U$. So, as $(X, \tau)$ is a d-space, there is $d\in D$ such that $d_0\le_{\tau^*} d$ and $\uparrow_{\tau^*}d=\uparrow_\tau d\subseteq D$.

All these together show that $(X, \tau^{*})$ is a d-space using the characterization given in \cite{li-yuan-zhao-2020}.
\end{proof}
\vskip 0.5cm
As $cl_{\tau^{*}}(\{x\})=cl_{\tau}(\{x\})\cup cl_{\tau}(\{y\})\not=cl_{\tau}(\{x\})$, the topology $\tau^*$ is strictly coarser than $\tau$. Hence we have the following conclusions.

\begin{corollary}\label{minimal sober is a chain}
If $(X, \tau)$ is a minimal sober space (well-filtered space, d-space, respectively), then $(X, \le_\tau)$ is a chain (for any $x, y\in X$, it holds that either $x\le_{\tau} y$ or $y\le_{\tau} x$).
\end{corollary}

Recall that the upper topology $\nu(P)$ on a poset $P$ is the topology of which $\{ P-\downarrow x: x\in P\}$ is a subbase \cite{Gier-2003}. The following result should have been proved by other people already. For reader's convenience, we give a brief proof.

\begin{lemma}\label{lem}
	For any chain $C$, $\nu(C)=\sigma(C)$.
\end{lemma}
\begin{proof}
	Clearly, $\nu(C)\subseteq \sigma(C)$.

    Let $F$ be a proper closed set of $(C,\sigma(C))$ and $A$ be the set of upper bound of $F$. Then $A\not =\emptyset$.

     Since $F$ is a lower set of a chain, for any $c\in C$, either $c\in F$ or $c\in A$.

       For any $y\in \bigcap_{x\in A}\downarrow x$,  if $y\not\in F$, then $y\in A$ and $y$ is the smallest element in $A$, thus $y=\sup F$. But $F$ is Scott closed, and $F$ is a directed set (every chain is a directed subset), thus $sup F\in F$, so $y\in F$, a contradiction.  So $y\in F$ must hold. It follows that $F=\bigcap_{x\in A}\downarrow x\in\Gamma(C,\nu(C))$. Therefore,  $\nu(C)=\sigma(C)$.
\end{proof}

One characterization of minimal $T_0$ spaces  was given  in  \cite[Theorem 1]{larson-1969}.

\begin{proposition}\label{s of T0}
	Let $(X,\tau)$ be a $T_0$ space, $\mathcal{A}=\{cl(\{x\}):x\in X\}$ and $\le_{\tau}$ be the specialization order of $(X,\tau)$. Then the following conditions are equivalence.
	\II
	\I[(1)] $(X,\tau)$ is a minimal $T_0$ space.
	\I[(2)] $\{X-A:A\in\mathcal{A}\}$ is a base of $\tau$ and $cl(F)\in\mathcal{A}$ for each $F\subseteq_{fin}X$.
	\I[(3)] $(X,\le_{\tau})$ is a chain and $\tau=\nu(X,\le_\tau)$ where $\nu(X,\le_\tau)$ is the upper topology on $(X, \le_{\tau})$.
	\III
\end{proposition}

A poset $P$ is sup-complete if for every nonempty subset $B\subseteq P$,  $\bigvee B$ exists. Thus a chain $P$ is sup-complete if and only if it is a dcpo.

\begin{lemma}\label{s of sober}
	Let $(X,\tau)$ be a $T_0$ space. If $(X,\le_{\tau})$ is a sup-complete chain and $\tau=\sigma(X,\le_{\tau})$, then $(X, \tau)$ is a minimal sober space (resp., well-filtered space, d-space).
\end{lemma}
\begin{proof} It is well-known that every sup-complete chain is a domain and the Scott space of each domain is sober \cite{Gier-2003}\cite{Goubault-2013}. Hence $(X, \tau)$ is sober.
	
By Lemma \ref{lem} and Proposition \ref{s of T0},  $(X,\tau)$ is a minimal $T_0$ space. Thus there is no $T_0$ topology strictly coarser than $\tau$. Hence $(X,\tau)$ is a minimal sober space.
Recall that every  sober space is well-filtered, every well-filtered space is a d-space and every d-space is $T_0$ \cite{Gier-2003}\cite{Goubault-2013}, it follows that $(X, \tau)$ is also a minimal well-filtered and minimal d-space.
\end{proof}

Now we have the main result of this section.

\begin{theorem}\label{e of sober}
	Let $(X, \tau)$ be a $T_0$ space. Then the following statements are equivalent.
	
	\II
	\I[(1)] $(X,\tau)$ is a minimal sober space.
	\I[(2)] $(X,\tau)$ is a minimal well-filtered space.
	\I[(3)] $(X,\tau)$ is a minimal d-space.
	\I[(4)] $(X,\le_\tau)$ is a sup-complete chain and $\tau=\sigma(X,\le_\tau)$.
	\III
\end{theorem}
\begin{proof}
By Lemma \ref{s of sober}, $(4)\Rightarrow (1),(2),(3)$. And	$(1),(2),(3)\Rightarrow (4)$, by Lemmas \ref{n of sober}, \ref{lem} and the fact that for any d-space (sober space, well-filtered space, respectively) $(X, \tau)$, $(X, \le_{\tau})$ is a dcpo (thus
$(X, \le_{\tau})$ is sup-complete if it is a chain).
\end{proof}

It is natural to wonder  whether for any sober topology $\tau$ on a set $X$,  there is a minimal sober topology $\mu \subseteq \tau$.
The answer is no, as shown by the following example.

\begin{example}
	Let $X=\mathbb{R}\cup \{\top\}$ and $\tau=\{\emptyset\}\cup\{X-F:F\subseteq_{fin}\mathbb{R}\}$.
	
	Obviously, $(X,\tau)$ is sober. Assume $\tau^*$ is a minimal sober topology with $\tau^*\subseteq \tau$. If $|cl_{\tau^*}(\{r\})|$ is infinite for some $r\in\mathbb{R}$, then $cl_{\tau^*}(\{r\})=X$ by $\tau^*\subseteq \tau$. It contradicts $X=cl_{\tau}(\{\top\})\subseteq cl_{\tau^*}(\{\top\})$ and $cl_{\tau^*}(\{r\})\not=cl_{\tau^*}(\{\top\})$ (because $\tau^*$ is $T_0$). Hence $|cl_{\tau^*}(\{r\})|$ is a finite number for each $r\in\mathbb{R}$.

 By Theorem \ref{e of sober}, $(X,\le_{\tau^*})$  is a chain. Hence $cl_{\tau^{*}}(\{r\})=\downarrow r=\{x\in X: x\le_{\tau^{*}} r\}$.

 It then follows that $|cl_{\tau^*}(\{r_1\})|=|cl_{\tau^*}(\{r_2\})|$ if and only if $r_1=r_2$. Thus  there is an injective map
 $f: \mathbb{R} \rightarrow \mathbb{N}$ such that $f(r)=|cl_{\tau^*}(r)|$. But it is impossible.
 Hence there is no minimal sober topology coarser than $\tau$.
\end{example}

At the end of this section, we prove that the soberification of a minimal $T_0$ space is a minimal sober space.

For any $T_0$ space $(X, \tau)$, the set $Irr_c(X,\tau)$ with the lower Vietoris topology $\tau^{*}=\{\diamondsuit U: U\in \tau\}$ is a sober space called the soberification of $(X, \tau)$. Here $\diamondsuit U=\{C\in Irr_c(X,\tau): C\cap U\not=\emptyset\}$.
 In addition, the specialization order $\le_{\tau^*}$ on $Irr_c(X,\tau)$ coincides with the inclusion order $\subseteq$ (see
\cite{Gier-2003}\cite{Goubault-2013}).

\begin{proposition}
	If  $(X, \tau)$ is a minimal $T_0$ space, then the soberification $(Irr_c(X,\tau),\tau^*)$ of $(X,\tau)$ is a minimal sober space.
\end{proposition}
\begin{proof} By Proposition \ref{s of T0}, $(X, \le_{\tau})$ is a chain. Also every closed set of $(X, \tau)$ is a lower set of $(X, \le_{\tau})$. In particular, every member of  $Irr_c(X,\tau)$ is a lower set. As $(X, \le_{\tau})$ is a chain, every two lower sets are comparable, it follows that $(Irr_c(X,\tau),\le_{\tau^*})$ is a chain.  Also $(Irr_c(X,\tau),\le_{\tau^*})$ is a dcpo, thus sup-complete.

Since  $\tau^*$ is sober, we have $$\nu(Irr_c(X,\tau),\le_{\tau^*})\subseteq\tau^*\subseteq \sigma(Irr_c(X,\tau),\le_{\tau^*}).$$ Then $\tau^{*}=\sigma(Irr_c(X,\tau),\le_{\tau^*})$, by Lemma \ref{lem}.

By Theorem \ref{e of sober},  $(Irr_c(X,\tau),\tau^*)$ is a minimal sober space.
\end{proof}

	\section{Summary and further work}
	In this paper we study the  sober topologies in the lattice of all topologies on a set. The main results are (1) every $T_1$ topology is the join of some sober topologies; (2) every topology is the meet of some sober topologies; (3) the set of all sober topologies is directed complete; (4) the minimal sober spaces are precisely the sup-complete chains equipped with the Scott topology.

There are still some problems on sober topologies deserve to be considered.
	
	 There are $T_0$ topologies which are not the join of sober topologies. Thus we have the following  problems:
	
	(1)   Which $T_0$  topologies are the joins of sober topologies?

    (2)  Which $T_0$  topologies are the joins of finite number of sober topologies?
	
	In \cite{steiner-1966}, Steiner proved that the lattice of topologies on each  set is complemented.
	Thus we have the following problem.
	
	(3) Which topologies have a  sober complement?
	
	Given a set $X$ with cardinality $\kappa$, people need to study the cardinality of all topologies on $X$ of certain types (such as $T_1$, Hausdorff, etc.)
	
	On sobriety, we have the following problem.
	
	(4) What is the cardinality of all sober topologies on a set $X$ with  $|X|=\kappa$?
	\vskip 1cm
	\section*{Acknowledgements}
	This work is supported by the National Nature Science Foundation of China (Grant No.:12231007).

    \section*{Declaration of interests} The authors declare that they have no known competing financial interests
or personal relationships that could have appeared to influence the work reported in this paper.

	\section*{Reference}
	\bibliographystyle{plain}

    \end{document}